\documentclass[12pt]{amsart}
\usepackage{amsmath}

\newcommand{\pd}{\partial}

\newcommand{\R}{\mathbb{R}}
\newcommand{\uS}{\mathbb{S}^{n-1}}

\newcommand{\MA}{Monge-Amp\`ere }

\newtheorem{theorem}{Theorem}[section]
\newtheorem{lemma}[theorem]{Lemma}

 \theoremstyle{definition}
\newtheorem{definition}[theorem]{Definition}

\theoremstyle{remark}
\newtheorem{remark}[theorem]{Remark}

\numberwithin{equation}{section}

\begin{document}

\title{The $L_p$ Gauss image problem}

\author{Chuanxi Wu}
\address{Faculty of Mathematics and Statistics, Hubei Key Laboratory of Applied Mathematics, Hubei University,  Wuhan 430062, P.R. China}

\email{cxwu@hubu.edu.cn}

\author{Di Wu}
\address{Faculty of Mathematics and Statistics, Hubei Key Laboratory of Applied Mathematics, Hubei University,  Wuhan 430062, P.R. China}
\email{wudi19950106@126.com}

\author{Ni Xiang}
\address{Faculty of Mathematics and Statistics, Hubei Key Laboratory of Applied Mathematics, Hubei University, Wuhan 430062, P.R. China}
\email{nixiang@hubu.edu.cn}

\date{}

\begin{abstract}
In this paper we study the $L_p$ Gauss image problem, which is a
generalization of the $L_p$ Aleksandrov problem and the Gauss image problem in convex geometry.
We obtain the existence result for the $L_p$ Gauss image problem in two cases (i) $p>0$ or (ii) $p<0$ with the given even measures.
\end{abstract}

\maketitle {\it \small{{\bf Keywords}: the $L_p$ Gauss image problem,
  Surface area measure,
  Curvature measure,
  Existence of solution,
  Blaskchke selection theorem.}

{{\bf MSC 2020}: Primary 35J96, Secondary
52A20.}
}

\vskip4ex

%%%%%%%%%%%%%%%%%%%%%%%%%%%%%%%%%%%
\section{Introduction}

Let $\mathcal{K}^n_0$ denote the set of convex bodies in $\R^n$ which contains the
origin in its interior. For $K\in\mathcal{K}^n_0$, its support function, $h_K$, is defined by
\begin{equation*}
h_K(x) := \max_{y\in K} \langle y,x \rangle, \quad \forall\, x\in\uS,
\end{equation*}
where $\langle y,x \rangle$ is the standard inner product of $x$ and $y$ in $\R^n$.

Suppose $K,L\in\mathcal{K}^n_0$ and $t\geq0$, the Minkowski combination, $K+tL\in\mathcal{K}^n_0$, is given by
\begin{equation*}
h_{K+tL}=h_K+t h_L;
\end{equation*}
for negative $t$, $K+tL$ can be defined if $|t|$ is sufficient small.
Aleksandrov's variational formula shows that
\begin{equation*}
\frac{d}{dt}V(K+tL)\bigg|_{t=0}=\int_{\uS}h_L(x)dS(K,x), \quad \forall L\in\mathcal{K}^n_0,
\end{equation*}
where $V(\cdot)$ denotes the volume functional and $S(K,\cdot)$ is the surface area measure.

A nature extension is $L_p$ Minkowski sum, which was first defined by Firey in case $p\geq1$.
Suppose $K,L\in\mathcal{K}^n_0$ and $t\geq0$, the $L_p$ Minkowski combination, $K+_p t\cdot L\in\mathcal{K}^n_0$, is defined by
\begin{equation*}
h^p_{K+_p t\cdot L}=h^p_K+t h^p_L.
\end{equation*}
In the early 1900's, $K+_p t\cdot L$ can be defined for negative $t$ if $|t|$ is sufficient small.  Lutwak \cite{LW.JDE.254-2013.983} showed the variational formula
\begin{equation*}
\frac{d}{dt}V(K+_p t\cdot L)\bigg|_{t=0}=\frac{1}{p}\int_{\uS}h^p_L(x)dS_p(K,x),\quad \forall L\in\mathcal{K}^n_0,
\end{equation*}
where $S_p(K,\cdot)$ is the $L_p$ surface area measure, satisfying
\begin{equation}\label{sp}
dS_p(K,\cdot)=h_K^{1-p}dS(K,\cdot).
\end{equation}
Thus by \eqref{sp}, the $L_p$ surface area measure can be defined for any $p\in\R$.

Suppose $K\in\mathcal{K}^n_0$, $K^\ast$ stands for the polar body of $K$, which is given by
\begin{equation}\label{ast}
K^\ast=\bigcap_{y\in K}\{x\in \mathbb{R}^n: \langle x, y\rangle\leq1\}.
\end{equation}

By combining the concept of $L_p$ Minkowski combination with that of polarity, we obtain another kind of combination, the $L_p$ harmonic combination, $K\hat{+}_p t\cdot L\in\mathcal{K}^n_0$, is defined by
\begin{equation*}
K\hat{+}_p t\cdot L=(K^*+_p t\cdot L^*)^*.
\end{equation*}
Huang, Lutwak, Yang and Zhang \cite{HLYZ.JDG.110-2018.1} considered the entropy functional
\begin{equation}\label{e}
\mathcal{E}(K)=-\int_{\uS}\log h_K(x)dx,
\end{equation}
and got the variational formula
\begin{equation*}
\frac{d}{dt}\mathcal{E}(K\hat{+}_p t\cdot L)\bigg|_{t=0}=\frac{1}{p}\int_{\uS}\rho^{-p}_L(u)dJ_p(K,u),
\quad \forall L\in\mathcal{K}^n_0,
\end{equation*}
where $\rho_L$ is the radial function of $L$, given by
\begin{equation*}
\rho_L(u) :=\max\{t>0 : t u\in L\}, \quad\forall\, u\in\uS.
\end{equation*}
They proved the $L_p$ integral curvature $J_p(K,\cdot)$ satisfies
\begin{equation}\label{J_p}
dJ_p(K,\cdot)=\rho_K^{p}dJ(K,\cdot),
\end{equation}
where $J(K,\cdot)$ is the Aleksandrov integral curvature.

For any Borel measurable subset $\omega\subset\uS$,  the radial Gauss image of $\omega$,  $\boldsymbol{\alpha}_K(\omega)$, is given by
\begin{equation*}
\boldsymbol{\alpha}_K(\omega)=\{x\in\mathbb{S}^{n-1}:  \langle\rho_K(u)u,x\rangle=h_K(x)\text{ for some }u\in \omega\}.
\end{equation*}
Then $J(K,\cdot)$ is the spherical Lebesgue measure of the radial Gauss image $\boldsymbol{\alpha}_{K}$, that is,
\begin{equation}\label{j}
J(K,\omega)=\mathcal{H}^{n-1}(\boldsymbol{\alpha}_{K}(\omega)).
\end{equation}

Recently, Boroczky, Lutwak, Yang, Zhang and Zhao \cite{B2020}
proposed the Gauss image measure $\lambda(K,\omega)$, which is a generalization of the Aleskandrov integral measure. We state it more precisely below.

\begin{definition}
Let $\lambda$ be a Borel measure on $\uS$, and $K\in\mathcal{K}^n_0$.
Then the Gauss image measure of $\lambda$ via $K$ is defined by
\begin{equation}\label{lambda}
\lambda(K,\omega)=\lambda(\boldsymbol{\alpha}_{K}(\omega)),
\end{equation}
where $\omega$ is a Lebesgue measurable subset of $\uS$.
\end{definition}

Based on \eqref{e}, \eqref{j} and \eqref{lambda}, it is nature to consider the functional
\begin{equation}\label{g}
G(K)=-\int_{\uS}\log h_K(x)d\lambda(x).
\end{equation}
And its variational formulas is
\begin{equation*}
\frac{d}{dt}G(K\hat{+}_p t\cdot L)\bigg|_{t=0}=\frac{1}{p}\int_{\uS}\rho_{L}^{-p}(u)d\lambda_p(K,u),
\quad \forall L\in\mathcal{K}^n_0,
\end{equation*}
where $\lambda_p(K,\cdot)$ is the $L_p$ Gauss image measure, satisfying
\begin{equation}\label{lambda_p}
d\lambda_p(K,\cdot)=\rho_K^p d\lambda(K,\cdot).
\end{equation}
It is clear that $\lambda_0(K,\cdot)=\lambda(K,\cdot)$, and \eqref{lambda_p} becomes \eqref{J_p} if $\lambda(K,\cdot)$ is $J(K,\cdot)$. Therefore, the $L_p$ Aleksandrov measure and the Gauss image measure
are special cases of the $L_p$ Gauss image measure.

\textbf{The $L_p$ Gauss image problem.} \emph{For a fixed $p\in\R$, suppose $\lambda$ and $\mu$ are two Borel measures defined on the Borel measurable subsets of $\uS$. What are the necessary and sufficient conditions, on $\lambda$ and $\mu$, such that there exists a convex
body $K$,
\begin{equation*}
\mu=\lambda_p(K,\cdot)
\end{equation*}
on the Borel subsets of $\uS$? And if such a body exists, to what
extent is it unique? }

When $\lambda$ is spherical Lebesgue measure, the $L_p$ Gauss image
problem is just the $L_p$ Aleksandrov problem, see \cite{A1, A2, B16,HLYZ.JDG.110-2018.1}. The $L_0$ Gauss image problem is just the Gauss image problem which was first mentioned in \cite{B2020}, and the existence of smooth solution for the Gauss image problem was in \cite{CWX}.
It is necessary to contrast the $L_p$ Gauss image problem with the various Minkowski
problems and dual Minkowski problems that have been extensively
studied, see \cite{BLYZ.JAMS.26-2013.831,
  CLZ.TAMS.371-2019.2623,
  CW.AIHPANL.17-2000.733,
  CW.Adv.205-2006.33,
  HLYZ.DCG.33-2005.699,
  JLW.JFA.274-2018.826,
  JLZ.CVPDE.55-2016.41,
  Lu.SCM.61-2018.511,
  Lu.JDE.266-2019.4394,
  LW.JDE.254-2013.983,
  Lut.JDG.38-1993.131,
  LYZ.TAMS.356-2004.4359,
  Sta.Adv.167-2002.160,
   Zhu.Adv.262-2014.909,Zhu.JDG.101-2015.159}
for the $L_p$-Minkowski problem, \cite{BHP.JDG.109-2018.411,
  HP.Adv.323-2018.114,
  HJ.JFA.277-2019.2209,
  HLYZ.Acta.216-2016.325,
  LSW.JEMSJ.22-2020.893,
  Zha.CVPDE.56-2017.18,
  Zha.JDG.110-2018.543}
for the dual Minkowski problem,
\cite{BF.JDE.266-2019.7980,
  CHZ.MA.373-2019.953,
  CCL,
  HLYZ.JDG.110-2018.1,
  HZ.Adv.332-2018.57,
  LLL,LYZ.Adv.329-2018.85}
for the $L_p$ dual Minkowski problem,
\cite{BBC.AiAM.111-2019.101937,
  HLYZ.Adv.224-2010.2485,
  HH.DCG.48-2012.281,
  JL.Adv.344-2019.262} for the Orlicz Minkowski problem,
\cite{CLLN, CTWX,GHW+.CVPDE.58-2019.12,GHXY.CVPDE.59-2020.15, LL.TAMS.373-2020.5833} for the dual Orlicz Minkowski problem, \cite{FB} for the Orlicz Aleskandrov problem.

The existence result for the $L_p$ Gauss image problem when $p>0$ will be presented in the followong. The idea goes back as for as \cite{HLYZ.JDG.110-2018.1}.

\begin{theorem}\label{main1}
Suppose $p>0$. If $\lambda$ and $\mu$ are two finite Borel measures on $\uS$ which satisfy

(1) $\lambda$ is absolutely continuous with respect to spherical Lebesgue measure;

(2) $\lambda(A)>0$ for any nonempty open set $A\subset\uS$;

(3) $\mu$ is not concentrated in any closed hemisphere of $\uS$.\\
Then there exists a convex body $K$ such that $\mu=\lambda_p(K,\cdot)$.
\end{theorem}

If the given measures are even, the third condition in Theorem \ref{main1} is naturally satisfied, and we can get the existence result for $p<0$.

\begin{theorem}\label{main2}
Suppose $p<0$. If $\lambda$ and $\mu$ are two finite, even Borel measures on $\uS$ which satisfy

(1) $\lambda$ is absolutely continuous with respect to spherical Lebesgue measure;

(2) $\lambda(A)>0$ for any nonempty open set $A\subset\uS$;

(3) $\mu$ is vanishes on great sub-spheres of $\uS$.\\
Then there exists an origin-symmetric convex body $K$ such that $\mu=\lambda_p(K,\cdot)$.
\end{theorem}

In particular, when $\mu$ has a density, say $f$, and $\lambda$ has a density, say $g$, the $L_p$ Gauss image problem asks: Under what conditions on the two given functions $f,g:\uS\rightarrow[0,\infty)$,  does there exist a solution $h:\uS\rightarrow (0,\infty)$, that is the support function of a convex body $K^*$, to the \MA equation
\begin{equation}\label{1.9}
g\bigg(\frac{\nabla h+hx}{|\nabla h+hx|}\bigg)\frac{h^{1-p}}{(|\nabla h|^2+h^2)^{\frac{n}{2}}}\det(\nabla^2 h+hI)=f(x) \quad \text{on} \quad \uS.
\end{equation}

\begin{remark}
The weak solution of the equation \eqref{1.9} for $g\equiv1$ was solved by Huang, Lutwak, Yang and Zhang \cite{HLYZ.JDG.110-2018.1}.
\end{remark}

It is worth pointing out that if $\lambda$ and $\mu$ have densities, the weak solution of the equation
\eqref{1.9} can be obtained according to Theorem \ref{main1} and Theorem \ref{main2}.

\begin{remark}
Assume $g>0$.

(1) For $p>0$, if for any hemisphere $\Theta$,  $$\int_{\Theta}f(x)dx>0,$$ then the equation \eqref{1.9} has a strictly positive solution.

(2) For $p<0$, if $f$ and $g$ are even functions, and if $$\int_{\uS}f(x)dx>0,$$
then the equation \eqref{1.9} has a strictly positive even solution.
\end{remark}

This paper is organized as follows. In section 2, we give some basic
knowledge about convex body. In
section 3, The $L_p$ Gauss image measure will be given, based on the radial Gauss image. In section 4, the variational formulas will be obtained and in the last section, Theorem \ref{main1} and Theorem \ref{main2} will be proved.

\section{Preliminaries}

In this section, some notions and facts will be set up, the details can be found in \cite{SR}.

Let $K$ be a convex body in $\R^n$, which means $K$ is a compact, convex subset in $\R^n$ with non-empty interior. And $\mathcal{K}^n=\{K: K$ is a convex body in $\R^n\}$,  $\mathcal{K}^n_e=\{K\in\mathcal{K}^n: K$ is origin-symmetric $\}$.

Suppose $K\in\mathcal{K}^n_0$, recall $h_K$ and $\rho_K$ denote the support function and the radial function of $K$, respectively,
\begin{equation}\label{h}
h_K(x) := \max_{y\in K} \langle y,x \rangle, \quad \forall\, x\in\uS,
\end{equation}
and
\begin{equation}\label{rho}
\rho_K(u) :=\max\{t>0 : t u\in K\}, \quad\forall\, u\in\uS.
\end{equation}
Note that
\begin{equation}\label{partial}
\partial K=\{ \rho_K(u)u: u\in\uS\}.
\end{equation}

By \eqref{h}, \eqref{rho} and \eqref{partial}, the support function $h_K$ and the radial function $\rho_K$ have the following relationship:
\begin{equation*}
h_K(x)=\max_{u\in\uS}\langle u,x\rangle\rho_K(u),\quad x\in\uS,
\end{equation*}
and
\begin{equation*}
\rho_K(u)=\max_{x\in\uS}\langle u,x\rangle/h_K(x),\quad u\in\uS.
\end{equation*}
The definition of polar body, i.e.\eqref{ast}, shows that
\begin{equation}\label{dual}
\rho_K=1/h_{K^*},\quad h_K=1/\rho_{K^*}.
\end{equation}
From \eqref{dual}, it is clear that
\begin{equation}\label{self}
K^{**}=K.
\end{equation}

Let $\Omega\subset\uS$ be a closed set that is not contained in any closed hemisphere of $\uS$. And $h,\rho:\Omega\rightarrow(0,\infty)$ are continuous functions. The Wulff shape determined by $h$, is denoted by
\begin{equation*}
[h]=\bigcap_{x\in\Omega}\{y\in\R^n:x\cdot y\leq h(x)\},
\end{equation*}
and the convex hull $\langle\rho\rangle$ generated by $\rho$, is denoted by
\begin{equation*}
\langle\rho\rangle=\text{conv}\{\rho(u)u:u\in\Omega\}.
\end{equation*}
And a useful fact is that, see \cite{HLYZ.Acta.216-2016.325},
\begin{equation}\label{2.5}
[h]^*=\langle 1/h \rangle.
\end{equation}

If $h_K$ is the support function of a convex body $K$, then
\begin{equation*}
[h_K]=K,
\end{equation*}
and if $\rho_K$ is the radial function of a convex body $K$, then
\begin{equation}\label{2.6}
\langle\rho_K\rangle=K.
\end{equation}

Assume $h_t:\Omega\rightarrow (0,\infty)$ is a continuous function defined for $t\in(-\delta,\delta)$ by
\begin{equation}\label{ht}
\log h_t(x)=\log h(x)+tf(x)+o(t,x),
\end{equation}
where $f:\Omega\rightarrow\R$ is continuous, $\delta>0$ and $o(t,\cdot):(-\delta,\delta)\times \uS\rightarrow \R$ is continuous and for $x\in\Omega$, $\lim_{t\rightarrow0}\frac{o(t,x)}{t}=0$.
We shall write $[h_t]$ as $[h,f,t]$. And if $h$ is the support function of a convex body $K$, write $[h_t]$ as $[K,f,t]$.

Suppose $\rho_t:\Omega\rightarrow (0,\infty)$ is a continuous function defined for $t\in(-\delta,\delta)$ by
\begin{equation}\label{rt}
\log \rho_t(u)=\log \rho(u)+tg(u)+o(t,u),
\end{equation}
where $g:\Omega\rightarrow\R$ is continuous and $o(t,\cdot):(-\delta,\delta)\times \uS\rightarrow \R$ is continuous and for $u\in\Omega$, $\lim_{t\rightarrow0}\frac{o(t,u)}{t}=0$. We shall write $\langle\rho_t\rangle$ as $\langle\rho,g,t\rangle$. And if $\rho$ is the radial function of a convex body $K$, write $\langle\rho_t\rangle$ as $\langle K,g,t\rangle$.

The $L_p$ Minkowski sum is the key role in the $L_p$ Brunn-Minkowski theory.
For fixed $p\in\R$, suppose $K$, $L\in\mathcal{K}^n_0$ and $a$, $b\geq0$, the $L_p$ Minkowski combination, $a\cdot K +_p b\cdot L\in \mathcal{K}^n_0$, is defined by the Wulff shape
\begin{equation*}
a\cdot K +_p b\cdot L=[(ah_K^p+bh_L^p)^\frac{1}{p}],\quad p\neq0,
\end{equation*}
and for $p=0$,
\begin{equation*}
a\cdot K +_0 b\cdot L=[h_K^ah_L^b].
\end{equation*}
The Wulff shape allows us to consider the case $a$ or $b$ is negative, with strictly positive $ah_K^p+bh_L^p$. The $L_p$ harmonic combination, $a\cdot K\hat{+}_p b\cdot L\in\mathcal{K}^n_0$, is defined by
\begin{equation*}
a\cdot K\hat{+}_p b\cdot L=(a\cdot K^*+_p b\cdot L^*)^*.
\end{equation*}

If $\mu$ is a fixed non-zero finite Borel measure on $\uS$,
then we define
\begin{equation}\label{mup}
||f:\mu||_p = \bigg( \frac{1}{|\mu|}\int_{\uS} f^pd\mu\bigg)^\frac{1}{p},\quad p\neq0,
\end{equation}
and
\begin{equation}\label{mu0}
||f:\mu||_0 =\exp\bigg( \frac{1}{|\mu|}\int_{\uS} \log f d\mu\bigg), \quad \forall f\in C^+(\uS).
\end{equation}

For any $x_0\in\uS$ and $0<\delta<1$, $\omega_\delta(x_0)$ and $\omega'_\delta(x_0)$ are defined by
\begin{equation}\label{2.10}
\omega_\delta(x_0)=\{u\in\uS:\langle u,x_0\rangle\geq\delta\},
\end{equation}
\begin{equation}\label{2.11}
\omega'_\delta(x_0)=\{u\in\uS:|\langle u,x_0\rangle|\geq\delta\}.
\end{equation}
It is obvious that
\begin{equation*}
\omega'_\delta(x_0)=\omega_\delta(x_0)\cup\omega_\delta(-x_0).
\end{equation*}
And $x_0^\bot$ denotes the hyperplane whose normal is $x_0$ and the origin $0\in x_0^\bot$.

\begin{definition}
Suppose $K$, $L\in\mathcal{K}^n_0$, their Hausdorff metric is given by
\begin{equation*}
d(K,L)=\max\{\sup_{x\in K}dist(x,L),\sup_{y\in L}dist(K,y)\}.
\end{equation*}
where $dist(x,L)=\inf_{y\in L}dist(x,y)$ and $dist(x,y)$ is the distance of $x$ and $y$ in $\R^n$.
We say $K_i\rightarrow K_0$, which means $d(K_i,K_0)\rightarrow0$ as $i\rightarrow\infty$.
\end{definition}

It is important to recall the following Lemma in \cite{HLYZ.JDG.110-2018.1}.

\begin{lemma}\label{lemma2.1}
Suppose $0<\delta<1$, $K_i$ is a sequence of convex bodies in $\mathcal{K}^n_0$, and if for some $x_0\in\uS$,
\begin{equation*}
\lim_{i\rightarrow\infty}h_{K_i}(x_0)=0,
\end{equation*}
then
\begin{equation*}
\lim_{i\rightarrow\infty}\rho_{K_i}(u)=0, \quad \forall u\in\omega_\delta(x_0).
\end{equation*}
\end{lemma}

\section{$L_p$ Gauss image problem}

In this section, we review some of the standard facts on the radial Gauss map $\alpha_K$ and reverse radial Gauss map $\alpha_K^*$.

Suppose $K\in\mathcal{K}^n_0$ and $x\in \uS$, the supporting hyperplane of $K$ in direction $x$ is given by
\begin{equation*}
H_K(x)=\{y\in \mathbb{R}^n:\langle x, y\rangle=h_K(x)\},
\end{equation*}
and $x$ is called the normal vector of $K$ at $y$.

For $\sigma\subset \partial K$, the spherical image of $\sigma$, $\boldsymbol{\nu}_K(\sigma):\partial K\rightarrow \uS$, is defined by
\begin{equation*}
\boldsymbol{\nu}_K(\sigma)=\{x\in\uS: y\in H_K(x) \text{ for some } y\in\sigma\}.
\end{equation*}

Let $\sigma_K\subset \pd K$ be the set consisting of $y\in\pd K$, for which the set $\boldsymbol{\nu}_K(\{y\})$ contains more than a single element.
Define the spherical image map
\begin{equation*}
\nu_K:\pd K\backslash\sigma_K\rightarrow\uS
\end{equation*}
such that $\nu_K(y)$ is the unique element of $\boldsymbol{\nu}_K(\{y\})$.

For $\eta\subset \uS$, the reverse spherical image of $\eta$, \textbf{y}$_K(\eta):\uS\rightarrow \partial K$, is defined by
\begin{equation*}
\textbf{y}_K(\eta)=\{y\in\partial K: y\in H_K(x) \text{ for some } x\in\eta\}.
\end{equation*}

The set $\eta_K$ is made up of $x\in\uS$, for which the set $\textbf{y}_K(\{x\})$ contains more than a single element.  Define the reverse spherical image map
\begin{equation*}
y_K:\uS\backslash\eta_K\rightarrow\pd K
\end{equation*}
such that $y_K(x)$ is the unique element of $\textbf{y}_K(\{x\})$. It is well known that $\mathcal{H}^{n-1}(\sigma_K)=\mathcal{H}^{n-1}(\eta_K)=0$ and $\nu_K$, $y_K$ are continuous functions,
the details can be found in \cite{SR}.

Suppose $K\in\mathcal{K}^n_0$, the radial map of $K$, $r_K:\uS\rightarrow \partial K$, is defined by
\begin{equation*}
r_K(u)=\rho_K(u)u,
\end{equation*}
its reverse map $r^{-1}_K:\partial K\rightarrow \uS$ is given by
\begin{equation*}
r^{-1}_K(y)=\frac{y}{|y|}.
\end{equation*}

With the above preparation, we can define the radial Gauss image and reverse radial Gauss image. Specifically, let $\omega$ and $\eta$ be subsets of $\uS$, the radial Gauss image of $\omega$, $\boldsymbol{\alpha}_K(\omega)$, is defined by
\begin{equation*}
\boldsymbol{\alpha}_K(\omega)=\{x\in\mathbb{S}^{n-1}:  \rho_K(u)u\in H_K(x)\text{ for some }u\in \omega\};
\end{equation*}
and the reverse radial Gauss image of $\eta$, $\boldsymbol{\alpha}^\ast_K(\eta)$, is defined by
\begin{equation*}
\boldsymbol{\alpha}^\ast_K(\eta)=\{u\in\mathbb{S}^{n-1}: \rho_K(u)u\in H_K(x) \text{ for some } x\in \eta\}.
\end{equation*}

Assume $K\in\mathcal{K}^n_0$, the radial Gauss map of $K$, $\alpha_K:\uS\backslash\omega_K\rightarrow\uS$, is given by
\begin{equation*}
\alpha_K=\nu_{K}(r_K),
\end{equation*}
where $\omega_K=r^{-1}_K(\sigma_{K})$, and the reverse radial Gauss map of $K$, $\alpha^*_K:\uS\backslash\eta_K\rightarrow\uS$, is given by
\begin{equation*}
\alpha^*_K=r^{-1}_K(y_K).
\end{equation*}

It was shown in \cite{HLYZ.Acta.216-2016.325} that the reverse radial Gauss
image of $K$ and the radial Gauss image of $K^*$ are identical, that is
\begin{equation*}
\boldsymbol{\alpha}_{K^\ast}=\boldsymbol{\alpha}_K^\ast.
\end{equation*}
And the definitions of $\boldsymbol{\alpha}_K$ and $\boldsymbol{\alpha}_K^*$ yield that
\begin{equation*}
\boldsymbol{\alpha}_{tK}=\boldsymbol{\alpha}_K,\quad \boldsymbol{\alpha}^*_{tK}=\boldsymbol{\alpha}^*_K, \quad \forall t>0.
\end{equation*}

Recall the Gauss image
measure $\lambda(K,\cdot)$ of $\lambda$ via $K$ is  defined by
\begin{equation*}
\lambda(K,\omega)=\lambda(\boldsymbol{\alpha}_K(\omega)).
\end{equation*}

When $\lambda$ is a Borel measure which is absolutely continuous with respect to the spherical Lebusgus measure, the Gauss image measure is
a Borel measure. The integral representation of $\lambda(K,\cdot)$ can be found in \cite{B2020}.

\begin{lemma}\label{lemma3.1}
Assume $\lambda$ is a Borel measure which is absolutely continuous with respect to spherical Lebesgue measure and $K\in\mathcal{K}^n_0$, then
\begin{equation}\label{3.8}
\int_{\uS}f(u)d\lambda(K,u)=\int_{\uS}f(\alpha^\ast_K(x))d\lambda(x), \quad \forall f\in C(\uS).
\end{equation}
\end{lemma}

\begin{proof}
Note that for any $\omega\subset\uS$ and $x\in\uS\backslash\omega_K$,
\begin{equation*}
\alpha_K^*(x)\in\omega \Leftrightarrow x\in\boldsymbol{\alpha}_K(\omega),
\end{equation*}
which implies that
\begin{equation}\label{3.2}
\chi_{\omega}(\alpha_K^*(x))=\chi_{\boldsymbol{\alpha}_K(\omega)}(x).
\end{equation}
We first show \eqref{3.8} holds for simple function $$\varphi=\sum_{i}c_i\chi_{\omega_i},$$ where $c_i\in\R$ and $\chi_{\omega_i}$ is the characteristic function of Borel subset $\omega_i\subset\uS$.

Since $\lambda$ is absolutely continuous with respect to spherical Lebesgue measure, \eqref{3.2} shows that
\begin{equation}\label{3.10}
\int_{\uS}\chi_{\omega}(\alpha_K^*(x))d\lambda(x)= \int_{\uS}\chi_{\alpha_K(\omega)}(x)d\lambda(x),
\end{equation}
then \eqref{lambda} and \eqref{3.10} mean
\begin{eqnarray*}
\int_{\uS}\varphi(u) d\lambda(K,u)&=& \int_{\uS}\sum_{i}c_i\chi_{\omega_i}(u)d\lambda(K,u)\\
&=& \sum_{i}c_i\lambda(K,\omega_i)\\
&=& \sum_{i}c_i\lambda(\boldsymbol{\alpha}_K(\omega_i))\\
&=& \int_{\uS}\sum_{i}c_i\chi_{\boldsymbol{\alpha}_K(\omega_i)}(x)d\lambda(x)\\
&=& \int_{\uS}\sum_{i}c_i\chi_{\omega}(\alpha_K^*(x))d\lambda(x)\\
&=& \int_{\uS}\varphi(\alpha_{K}^*(x))d\lambda(x).
\end{eqnarray*}
Let $f:\uS\rightarrow\R$ be a continuous function, then there exists a sequence of simple functions $\{\varphi_k\}$ such that $\varphi_k\rightarrow f$ as $k\rightarrow\infty$. By the dominated convergence theorem, it holds
\begin{equation*}
\int_{\uS}f(u)d\lambda(K,u)=\int_{\uS}f(\alpha_K^\ast(x))d\lambda(x).
\end{equation*}
This finishes the proof.
\end{proof}

Inspired by \cite{HLYZ.JDG.110-2018.1}, we define the $L_p$ Gauss image measure $\lambda_p(K,\cdot)$ of $K\in\mathcal{K}^n_0$: fixed $p\in\R$, $\lambda_p(K,\cdot)$ is a Borel measure and satisfies
\begin{equation}\label{3.11}
\int_{\uS}f(u)d\lambda_p(K,u)=\int_{\uS}f(\alpha_K^\ast(x))\rho_K^p(\alpha_K^\ast(x))d\lambda(x),\quad\forall f\in C(\uS).
\end{equation}
By \eqref{3.8} and \eqref{3.11}, it is easy to see
\begin{equation*}
d\lambda_p(K,\cdot)=\rho_K^p d\lambda(K,\cdot).
\end{equation*}

It is worth to point out the Gauss image measure as a functional from $\mathcal{K}^n_0$ to the space of Borel measures on $\uS$ is weakly convergent with respect to the Hausdorff metric, see Lemma 3.4 in \cite{B2020}.
\begin{lemma}\label{lemma3.2}
Assume $\lambda$ is a Borel measure which is absolutely continuous with respect to spherical Lebesgue measure, and $K_0,K_1,\cdots\in\mathcal{K}^n_0$ such that $K_i\rightarrow K_0$ as $i\rightarrow\infty$, then $\lambda(K_i,\cdot)\rightharpoonup \lambda(K_0,\cdot)$.
\end{lemma}

\section{Variational formulas for entropy of convex bodies}

In this section, variational formulas for the general entropy of convex bodies will be given. Recall the general entropy
\begin{equation*}
G(K)=-\int_{\uS}\log h_K(x)d\lambda(x),
\end{equation*}
and the dual general entropy, defined by
\begin{equation}\label{ed}
E(K)=\int_{\uS}\log\rho_K(u)d\lambda(u),
\end{equation}
it is clear that for any $K\in\mathcal{K}^n_0$,
\begin{equation}\label{gk}
E(K^*)=G(K).
\end{equation}

As will be shown, the following Lemma in \cite{HLYZ.Acta.216-2016.325} turns out to be a critical property.
\begin{lemma}\label{lemma4.1}
Let $\Omega\subset\uS$ be a closed set that is not contained in any closed
hemisphere of $\uS$. If $\rho_t$ is a logarithmic family of convex hulls of $\langle\rho_0,g,t\rangle$, then
\begin{equation*}
\lim_{t\rightarrow0}\frac{\log h_{\langle\rho_t\rangle}(x)-\log h_{\langle\rho_0\rangle}(x)}{t}=
g(\alpha_{\langle\rho_0\rangle}^*(x)),\quad \forall x\in \uS\backslash\omega_{\langle\rho_0\rangle}.
\end{equation*}
Furthermore, there exist $\delta>0$ and $M>0$, such that
\begin{equation*}
|\log h_{\langle\rho_t\rangle}(x)-\log h_{\langle\rho_0\rangle}(x)|\leq M|t|,\quad \forall (x,t)\in\uS\times (-\delta,\delta).
\end{equation*}
\end{lemma}

It is sufficient to make the following Lemma together with Lemma \ref{lemma3.1} and Lemma \ref{lemma4.1}.
\begin{lemma}\label{lemma4.2}
Assume $\lambda$ is a Borel measure which is absolutely continuous with respect to spherical Lebesgue measure. Let $K\in\mathcal{K}^n_0$ and $f,g:\uS\rightarrow\R$ be two continuous functions, if $\rho_t$ is given by \eqref{rt}, then
\begin{equation}\label{4.3}
\frac{d}{dt}G(\langle \rho_t \rangle)\bigg|_{t=0}=-\int_{\uS}g(u)d\lambda(K,u),
\end{equation}
and if $h_t$ is given by \eqref{ht}, then
\begin{equation}\label{4.4}
\frac{d}{dt}E([h_t])\bigg|_{t=0}=\int_{\uS}f(x)d\lambda(K^*,x).
\end{equation}
\end{lemma}

It remains to prove the variational formulas applying the Lemma \ref{lemma4.2} in this section.
\begin{lemma}
Assume $\lambda$ is a Borel measure which is absolutely continuous with respect to spherical Lebesgue measure and $K,L\in\mathcal{K}^n_0$. Thus for $p\neq0$,
\begin{equation}\label{4.5}
\frac{d}{dt}G(K\hat{+}_p t\cdot L)\bigg|_{t=0}=\frac{1}{p}\int_{\uS}\rho_{L}^{-p}(u)d\lambda_p(K,u),
\end{equation}
and for $p=0$,
\begin{equation}\label{4.6}
\frac{d}{dt}G(K\hat{+}_0 t\cdot L)\bigg|_{t=0}=-\int_{\uS}\log\rho_{L}(u)d\lambda(K,u).
\end{equation}
\end{lemma}

\begin{proof}
For $p\neq0$, let
\begin{equation*}
h_t=(h_K^p+th_L^p)^\frac{1}{p},
\end{equation*}
and choose $\delta>0$ such that
\begin{equation*}
\log h_t=\log h_K+\frac{1}{p} (\frac{h_L}{h_K})^p t+o_p(t,\cdot),
\end{equation*}
where $o_p:(-\delta,\delta)\times \uS\rightarrow \R$ is continuous and satisfies
\begin{equation*}
\lim_{t\rightarrow0}\frac{o_p(t,x)}{t}=0,\quad  \forall x\in\uS.
\end{equation*}
Choosing $$f=\frac{1}{p}(\frac{h_L}{h_K})^p,$$
we have
$$K+_pt\cdot L=[h_t]=[K,f,t].$$
Then \eqref{4.4} amounts to the fact that
\begin{eqnarray*}
\frac{d}{dt}E(K+_p t\cdot L)\bigg|_{t=0}&=&\frac{1}{p}\int_{\uS}(\frac{h_L}{h_K})^p(x)d\lambda(K^*,x).
\end{eqnarray*}
Replace $K,L$ by $K^*,L^*$ yields
\begin{eqnarray}\label{2.4}
\frac{d}{dt}E(K^*+_p t\cdot L^*)\bigg|_{t=0}&=&\frac{1}{p}\int_{\uS}(\frac{h_{L^*}}{h_{K^*}})^p(u)d\lambda(K^{**},u).
\end{eqnarray}
Based on \eqref{dual}, \eqref{self} and \eqref{2.4}, we can assert that
\begin{eqnarray}\label{4.8}
\frac{d}{dt}E(K^*+_p t\cdot L^*)\bigg|_{t=0}&=&\frac{1}{p}\int_{\uS}(\frac{\rho_{K}}{\rho_{L}})^p(u)d\lambda(K,u).
\end{eqnarray}
Thus \eqref{4.5} can be obtained by \eqref{lambda_p}, \eqref{gk} and \eqref{4.8}.

We now turn to the case $p=0$.
Set $h_t=h_K h_L^t$, thus we get $$\log h_t=\log h_K+t\log h_L.$$

Since the proof for the case $p=0$ is similar in spirit to the case $p\neq0$, we omit the detials here.
\end{proof}

\section{Existence of solutions to the $L_p$ Gauss image problem}

For given Borel measures $\lambda$ and $\mu$, and $p\neq0$, define the functional $\Phi_{\lambda,\mu,p}:C^+(\uS)\rightarrow \R$, for any $f\in C^+(\uS)$,
\begin{eqnarray}\label{phi}
\Phi_{\lambda,\mu,p}(f)&=&-\frac{1}{|\lambda|}\int_{\uS}\log h_{\langle f\rangle}(x)d\lambda(x)
-\frac{1}{p}\log(\frac{1}{|\mu|}\int_{\uS} f^{-p}(u)d\mu(u))\\ \nonumber
&=&G(\langle f\rangle)/|\lambda|+\log||f:\mu||_{-p}.
\end{eqnarray}
It is easy to check that $\Phi_{\lambda,\mu,p}(f)$ is homogeneous of degree 0, that is
\begin{equation}\label{5.2}
\Phi_{\lambda,\mu,p}(tf)=\Phi_{\lambda,\mu,p}(f), \quad \forall t>0,  f\in C^+(\uS)
\end{equation}

The maximization problem is:
\begin{equation}\label{5.3}
\sup\{\Phi_{\lambda,\mu,p}(f):f\in C^+(\uS)\}.
\end{equation}

The following Lemma shows that the solution of maximization problem \eqref{5.3} must be radial function of a convex body $K$.
\begin{lemma}\label{lemma5.1}
Suppose $p\neq0$, then a convex body $K\in\mathcal{K}^n_0$ is a solution of the maximization problem
\begin{equation*}
\sup\{\Phi_{\lambda,\mu,p}(\rho_K):K\in \mathcal{K}^n_0\}
\end{equation*}
if and only if $\rho_K$ is a solution of the maximization problem
\begin{equation*}
\sup\{\Phi_{\lambda,\mu,p}(f):f\in C^+(\uS)\}.
\end{equation*}
\end{lemma}

\begin{proof}
The convex hull is given by $$\langle f \rangle = \text{conv} \{f(u)u:u\in\uS\},\quad \forall f\in C^+(\uS),$$ then it is easily seen that
\begin{equation*}
\rho_{\langle f \rangle}\geq f,
\end{equation*}
which implies that
\begin{equation}\label{5.4}
||\rho_{\langle f \rangle}:\mu||_{-p}\geq ||f:\mu||_{-p}.
\end{equation}
By \eqref{2.6}, that is
\begin{equation*}
\langle \rho_{\langle f \rangle} \rangle = \langle f \rangle,
\end{equation*}
hence
\begin{equation}\label{5.5}
G(\langle \rho_{\langle f \rangle} \rangle)=G(\langle f \rangle).
\end{equation}
Applying \eqref{5.4} and \eqref{5.5} we conclude that
\begin{equation*}
\Phi_{\lambda,\mu,p}(f)\leq\Phi_{\lambda,\mu,p}(\rho_{\langle f \rangle}).
\end{equation*}
\end{proof}
We have divided the proof into a sequence of Lemmas.

\subsection{The proof of Theorem \ref{main1}}

\

In this subsection, we deal with the case $p>0$.
\begin{lemma}\label{lemma5.3}
Under the assumptions of Theorem \ref{main1}, there exists a convex body $K_0\in\mathcal{K}^n_0$ such that
\begin{equation*}
\sup\{\Phi_{\lambda,\mu,p}(\rho_K):K\in\mathcal{K}^n_0\}=\Phi_{\lambda,\mu,p}(\rho_{K_0}).
\end{equation*}
\end{lemma}

\begin{proof}
Let
\begin{equation*}
\mathcal{K}=\{K\in\mathcal{K}^n_0:\int_{\uS} h_K^pd\mu=|\lambda|\}.
\end{equation*}
Define the function $\varphi:\uS\rightarrow \R$, for $x\in\uS$,
$$\varphi(x)=\int_{\uS}(\langle x,u\rangle)_+^pd\mu(u).$$
Assume $\varphi$ attains its minimum at some vectors $x_{\mu}$, thus
\begin{equation*}
\int_{\uS}(\langle x,u\rangle)_+^pd\mu(u)\geq\int_{\uS}(\langle x_{\mu},u\rangle)_+^pd\mu(u)>0,
\end{equation*}
the strict inequality holds because $\mu$ is not concentrated in any closed hemisphere of $\uS$.

Now we claim that $\mathcal{K}$ is bounded. In fact, for any $K\in\mathcal{K}$. Assume $\rho_K$ attains its maximum at some $x_K\in\uS$, by the definition of the support function, it is clear that

\begin{equation*}
\rho_K(x_K)(\langle x_K,u\rangle)_+\leq h_K(u), \quad \forall u\in \uS.
\end{equation*}
Then
\begin{equation*}
\rho^p_K(x_K)\int_{\uS}(\langle x_K,u\rangle)^p_+ d\mu(u)\leq \int_{\uS} h^p_K(u)d\mu(u)=|\lambda|,
\end{equation*}
which implies that
\begin{equation*}
\rho_K(x_K)\leq |\lambda|^\frac{1}{p}\bigg(\int_{\uS}(\langle x_{\mu},u\rangle)_+^pd\mu(u)\bigg)^{-\frac{1}{p}}=c(\lambda,\mu,p).
\end{equation*}
Therefore, $\mathcal{K}$ is bounded.

By \eqref{5.2}, $\Phi_{\lambda,\mu,p}(\rho_K)$ is homogeneous of degree 0, it is feasible to choose a maximizing sequence $K_i$ for $\Phi_{\lambda,\mu,p}(\rho_K)$ such that $K_i^*\in\mathcal{K}$.
By Blaskchke selection theorem, the sequence $K_i^*$ has a subsequence, still denoted by $K_i^*$, such that $K_i^*\rightarrow K_0^*$ for some $K_0^*$.

We prove $K_0^*\in\mathcal{K}^n_0$ by contradiction. Or we can assume $K_0^*\subset u_0^\bot$ for some $u_0\in\uS$, then
$$\lim_{i\rightarrow\infty}h_{K_i^*}(u_0)=h_{K_0^*}(u_0)=0.$$
By Lemma \ref{lemma2.1}, there exists $0<\delta<1$ such that \begin{equation}\label{5.12}
\lim_{i\rightarrow\infty}\rho_{K_i^*}(x)\rightarrow0, \quad \forall x\in\omega_{\delta}(u_0).
\end{equation}
Then
\begin{eqnarray}\label{5.13}
G(K_i)&=&E(K_i^{*})\\\nonumber
&=&\int_{\uS}\log \rho_{K_i^*}(x)d\lambda(x)\\ \nonumber
&\leq&\int_{\omega_\delta(u_0)}\log \rho_{K_i^*}(x)d\lambda(x)+\int_{\uS\backslash \omega_\delta(u_0)}\log c(\lambda,\mu,p)d\lambda(x)\\ \nonumber
&\leq&\int_{\omega_\delta(u_0)}\log \rho_{K_i^*}(x)d\lambda(x)+\log c(\lambda,\mu,p)\lambda(\uS\backslash \omega_\delta(u_0)).
\end{eqnarray}
From \eqref{5.12},\eqref{5.13} and the condition $\lambda(\omega_\delta(u_0))$ is positive, it follows that
\begin{equation*}
G(K_i)\rightarrow-\infty \quad \text{as} \quad i\rightarrow\infty,
\end{equation*}
thus
\begin{equation*}
\Phi_{\lambda,\mu,p}(\rho_{K_i})\rightarrow-\infty \quad \text{as} \quad i\rightarrow\infty.
\end{equation*}

But it is easy to know that for $r=(|\lambda|/|\mu|)^{\frac{1}{p}}$, $rB\in\mathcal{K}$, and
\begin{equation*}
\Phi_{\lambda,\mu,p}(K_i)>\Phi_{\lambda,\mu,p}(rB)=\Phi_{\lambda,\mu,p}(B)=0,
\end{equation*}
which is a contradiction.
Then $K_0^*\in\mathcal{K}^n_0$, the proof is complete.
\end{proof}

\begin{lemma}\label{lemma5.2}
Under the assumptions of Theorem \ref{main1}, if $K$ is a solution of the maximization problem
\begin{equation}\label{5.6}
\sup\{\Phi_{\lambda,\mu,p}(\rho_K):K\in \mathcal{K}^n_0 \}
\end{equation}
under the restriction
\begin{equation}\label{lm}
\int_{\uS}\rho_K^{-p}d\mu=|\lambda|,
\end{equation}
then
\begin{equation*}
\mu=\lambda_p(K,\cdot).
\end{equation*}
\end{lemma}

\begin{proof}
By Lemma \ref{lemma5.1} and $K$ is a maximizer of \eqref{5.6},
\begin{equation*}
\Phi_{\lambda,\mu,p}(\rho_K)=\sup\{\Phi_{\lambda,\mu,p}(\rho_L):L\in \mathcal{K}^n_0\}=\sup\{\Phi_{\lambda,\mu,p}(f):f\in C^+(\uS)\}.
\end{equation*}
Define
\begin{equation}\label{5.7}
\rho_t=\rho_K e^{tg},\quad \forall g\in C^+(\uS),
\end{equation}
then $t=0$ must be the maximum point of $\Phi_{\lambda,\mu,p}(\rho_t)$, that is
\begin{equation}\label{5.8}
\frac{d}{dt}\Phi_{\lambda,\mu,p}(\rho_t)\bigg|_{t=0}=0.
\end{equation}
By \eqref{5.7},
\begin{equation*}
\log\rho_t=\log\rho_K + tg,
\end{equation*}
then Lemma \ref{lemma4.2} shows that
\begin{equation}\label{5.9}
\frac{d}{dt}G(\langle\rho_t\rangle)\bigg|_{t=0}=-\int_{\uS}g(u)d\lambda(K,u).
\end{equation}

For $|s|<1$, the inequality
\begin{equation}\label{5.10}
|e^s-1-s|\leq es^2
\end{equation}
holds. Choosing $s=-ptg(u)$, \eqref{5.10} shows that for $|t|<\frac{1}{|p|\max_{u\in\uS}g(u)}$,
\begin{equation*}
|e^{-ptg(u)}-1+ptg(u)|\leq ep^2g^2(u)|t|^2, \quad \forall u\in\uS,
\end{equation*}
which implies
\begin{equation*}
|\frac{e^{-tpg(u)}-1}{t}+pg(u)|\leq ep^2g^2(u)|t|, \quad \forall u\in\uS.
\end{equation*}
Then
\begin{equation*}
\lim_{t\rightarrow0}\frac{\rho_t^{-p}(u)-\rho_0^{-p}(u)}{t}= \lim_{t\rightarrow0}\frac{e^{-tpg(u)}-1}{t}\rho_K^{-p}(u)=-pg(u)\rho_{K}^{-p}(u), \quad \forall u\in\uS.
\end{equation*}
By \eqref{lm}, we get
\begin{eqnarray}\label{5.11}
|\lambda|\frac{d}{dt}\log||\rho_t:\mu||_{-p}\bigg| _{t=0}&=&
\int_{\uS}\rho_K^{-p}(u)d\mu(u)\cdot\frac{d}{dt}\log||\rho_t:\mu||_{-p}\bigg| _{t=0}\\\nonumber
&=&-\frac{1}{p}\int_{\uS}\lim_{t\rightarrow0}\frac{\rho_t^{-p}(u)
-\rho_0^{-p}(u)}{t}d\mu(u)\\\nonumber
&=&\int_{\uS}\rho_{K}^{-p}(u)g(u)d\mu(u).
\end{eqnarray}
Therefore, \eqref{5.8}, \eqref{5.9}, \eqref{5.11} and the definition of $\Phi_{\lambda,\mu,p}$ show that
\begin{equation*}
\int_{\uS}\rho_{K}^{-p}(u)g(u)d\mu(u)=\int_{\uS}g(u)d\lambda(K,u).
\end{equation*}
Thus $\rho_{K}(u)^{-p}d\mu(u)=d\lambda(K,u)$, which completes the proof in view of \eqref{lambda_p}.
\end{proof}

\textbf{The proof of Theorem \ref{main1}}. By Lemma \ref{lemma5.3}, there exists a convex body $K_0\in\mathcal{K}^n_0$ such that
\begin{equation*}
\sup\{\Phi_{\lambda,\mu,p}(\rho_K):K\in\mathcal{K}^n_0\}=\Phi_{\lambda,\mu,p}(\rho_{K_0}).
\end{equation*}
By Lemma \ref{lemma5.2}, $K_0$ satisfies that
\begin{equation*}
\mu=\lambda_p(K_0,\cdot).
\end{equation*}

\subsection{The proof of Theorem \ref{main2}}

\

This subsection is intended to provide a detailed proof of Theorem \ref{main2}.

\begin{lemma}\label{lemma5.5}
Under the assumptions of Theorem \ref{main2}, there exists an origin-symmetric convex body $K_0$ such that
\begin{equation*}
\sup\{\Phi_{\lambda,\mu,p}(\rho_K):K\in\mathcal{K}^n_e\}=\Phi_{\lambda,\mu,p}(\rho_{K_0}).
\end{equation*}
\end{lemma}

\begin{proof}
Let
\begin{equation*}
\mathcal{K}=\{K\in\mathcal{K}^n_e:\int_{\uS} \log h_K(x)d\lambda(x)=0\}.
\end{equation*}
Since $\Phi_{\lambda,\mu,p}$ is homogeneous of degree 0, the maximization problem is equivalent to
\begin{equation*}
\sup\{\log||\rho_K:\mu||_{-p}:K\in\mathcal{K}\}.
\end{equation*}

Now we claim that $\mathcal{K}$ is bounded. In fact, for any $K\in\mathcal{K}$. Assume $\rho_K$ attains its maximum at some $u_K\in\uS$, by the definition of the support function and $K\in\mathcal{K}^n_e$, it is clear that
\begin{equation*}
\rho_K(u_K)|\langle u_K,x\rangle|\leq h_K(x), \quad  \forall x\in \uS,
\end{equation*}
which implies
\begin{equation*}
\rho_{K}(u_K)|\lambda|+ \int_{\uS}\log |\langle u_K,x\rangle|d\lambda(x)\leq \int_{\uS}\log h_K(x)d\lambda(x)=0.
\end{equation*}
Then
\begin{equation*}
\rho_{K}(u_K)\leq -|\lambda|^{-1} \int_{\uS}\log |\langle u_K,x\rangle|d\lambda(x)= R.
\end{equation*}
Therefore, $\mathcal{K}$ is bounded.

By \eqref{5.2}, $\Phi_{\lambda,\mu,p}(\rho_K)$ is homogeneous of degree 0, it is feasible to choose a maximizing sequence $K_i$ for $\Phi_{\lambda,\mu,p}(\rho_K)$ such that $K_i\in\mathcal{K}$.
By Blaskchke selection theorem, the sequence $K_i$ has a subsequence, still denoted by $K_i$, such that $K_i\rightarrow K_0$ for some $K_0$.

We prove $K_0\in\mathcal{K}^n_e$ by contradiction. Otherwise, assume $K_0\subset x_0^\bot$ for some $x_0\in\uS$, then
$$\lim_{i\rightarrow\infty}h_{K_i}(x_0)=h_{K_0}(x_0)=0.$$
By Lemma \ref{lemma2.1}, there exists $0<\delta<1$ such that
\begin{equation*}
\lim_{i\rightarrow\infty}\rho_{K_i}(u)\rightarrow0, \quad \forall u\in\omega'_{\delta}(x_0).
\end{equation*}

Therefore,
\begin{eqnarray*}
\int_{\uS}\rho_{K_i}^{-p}(u)d\mu(u)&=&\int_{\omega'_{\delta}(x_0)}\rho_{K_i}^{-p}(u)d\mu(u)
+\int_{\uS\backslash\omega'_{\delta}(x_0)}\rho_{K_i}^{-p}(u)d\mu(u)\\
&\leq& \int_{\omega'_{\delta}(x_0)}\rho_{K_i}^{-p}(u)d\mu(u)+R^{-p}\mu(\uS\backslash\omega'_{\delta}(x_0)).
\end{eqnarray*}

Since $\mu$ vanishes on any great sub-sphere of $\uS$, then $\mu(\uS\cap x_0^\bot)=0$. Choose $1>\delta_1>\delta_2>\cdots>\delta_j\rightarrow0$, it is clear that
$$\uS\backslash\omega'_{\delta_1}(x_0)\supset\uS\backslash\omega'_{\delta_2}(x_0)\supset\cdots,$$
with
\begin{equation*}
\bigcap_{j=1}^{\infty}(\uS\backslash\omega'_{\delta_j}(x_0))=\uS\cap x_0^\bot.
\end{equation*}
Therefore
\begin{equation*}
\lim_{j\rightarrow\infty}\mu(\uS\backslash\omega'_{\delta_j}(x_0))=\mu(\uS\cap x_0^\bot)=0.
\end{equation*}
For any given $\epsilon>0$, there exists a positive integer $j_0$, such that
\begin{equation*}
R^{-p}\mu(\uS\backslash\omega'_{\delta_{j_0}}(x_0))<\frac{\epsilon}{2}.
\end{equation*}
Since $\rho_{K_i}\rightarrow 0$ uniformly on $\omega'_{\delta_{j_0}}(x_0)$ and $\mu$ is a finite measure, there exists a positive integer $N$, such that
\begin{equation*}
\int_{\omega'_{\delta_{j_0}}(x_0)}\rho_{K_i}^{-p}(u)d\mu(u)<\frac{\epsilon}{2}, \quad \forall i>N.
\end{equation*}
Therefore,
\begin{equation*}
\int_{\uS}\rho_{K_i}^{-p}(u)d\mu(u)\rightarrow0 \quad \text{ as } \quad i\rightarrow\infty.
\end{equation*}
So
\begin{equation*}
\Phi_{\lambda,\mu,p}(\rho_{K_i})=\log||\rho_{K_i}:\mu||_{-p}\rightarrow-\infty \quad \text{ as } \quad i\rightarrow\infty.
\end{equation*}

But $K_i$ is a maximizing sequence for $\Phi_{\lambda,\mu,p}$, then for sufficiently large $i$, one can see
\begin{equation*}
\Phi_{\lambda,\mu,p}(\rho_{K_i})>\Phi_{\lambda,\mu,p}(\rho_{B})=0.
\end{equation*}
That is a contradiction. Then $K\in\mathcal{K}^n_e$, the proof is complete.
\end{proof}

\begin{lemma}\label{lemma5.4}
Under the assumptions of Theorem \ref{main2}, if $K$ is a solution of the maximization problem
\begin{equation*}
\sup\{\Phi_{\lambda,\mu,p}(\rho_K):K\in \mathcal{K}^n_e \}
\end{equation*}
under the restriction
\begin{equation*}
\int_{\uS}\rho_K^{-p}d\mu=|\lambda|,
\end{equation*}
then
\begin{equation*}
\mu=\lambda_p(K,\cdot).
\end{equation*}
\end{lemma}

\begin{proof}
The proof of this Lemma is the same as that of Lemma \ref{lemma5.2}.
\end{proof}

\textbf{The proof of Theorem \ref{main2}}. By Lemma \ref{lemma5.5}, there exists a convex body $K_0\in\mathcal{K}^n_e$ such that
\begin{equation*}
\sup\{\Phi_{\lambda,\mu,p}(\rho_K):K\in\mathcal{K}^n_e\}=\Phi_{\lambda,\mu,p}(\rho_{K_0}).
\end{equation*}
By Lemma \ref{lemma5.4}, $K_0$ satisfies that
\begin{equation*}
\mu=\lambda_p(K_0,\cdot).
\end{equation*}

%%%%%%%%%%%%%%%%%%%%%%%%%%%%%%%%%%%%%%%%%%%
% \bibliographystyle{siam}
% \bibliography{article}

\end{document}